\documentclass[11pt]{article}

\usepackage[margin=1in]{geometry}
\usepackage{amsmath,amssymb,amsthm}
\usepackage{hyperref}

\newtheorem{theorem}{Theorem}

\newtheorem{claim}[theorem]{Claim}

\title{Log-concavity of subsequence counts of words}
\author{Vincent Vatter%
	\thanks{Department of Mathematics, University of Florida, Gainesville, Florida, USA. Email: \texttt{vatter@ufl.edu}.}%
}

\date{}

\begin{document}

\maketitle

\noindent For a word $w$ over a finite alphabet, let $\binom{w}{k}$ denote the number of distinct words arising as length-$k$ subsequences of $w$. In~1976, Chase~\cite{chase:subsequence-num:} proved that the sequence $\binom{w}{0}, \binom{w}{1}, \ldots$ is log-concave. His proof uses a triangular array indexed by prefixes of $w$ together with a meticulous analysis of ratios of several sums. We decompose by first letter instead, reducing the proof to a weighted average.

\begin{theorem}[Chase~\cite{chase:subsequence-num:}]
For every word $w$ and every $k\ge 1$, $\binom{w}{k}^2 \ge \binom{w}{k-1}\binom{w}{k+1}$.
\end{theorem}
\begin{proof}
For a letter $\ell$ occurring in $w$, let $w_\ell$ denote the suffix of $w$ after its first occurrence of $\ell$, the \emph{$\ell$-tail of $w$}. Every nonempty subsequence of $w$ is determined by its first letter $\ell$ together with a subsequence of $w_\ell$, so
\[
	\binom{w}{k+1} = \sum_\ell \binom{w_\ell}{k},
\]
the sum running over distinct letters $\ell$ of $w$. 
Writing $\rho_k(w) = \binom{w}{k}/\binom{w}{k-1}$, taken to be $0$ whenever the numerator is $0$, log-concavity is equivalent to $\rho_{k+1}(w) \le \rho_k(w)$. This is immediate when $\binom{w}{k} = 0$. Otherwise, applying the identity above to both numerator and denominator expresses $\rho_{k+1}(w)$ as a weighted average:
\[
	\rho_{k+1}(w)
	=
	\frac{\sum_\ell \binom{w_\ell}{k}}{\sum_\ell \binom{w_\ell}{k-1}}
	=
	\frac{\sum_\ell \binom{w_\ell}{k-1}\, \rho_k(w_\ell)}{\sum_\ell \binom{w_\ell}{k-1}}. \tag{$\dagger$}
\]

\begin{claim}
If $w = uv$ for words $u$ and~$v$, then $\rho_k(v) \le \rho_k(w)$.
\end{claim}
Suppose that $u$ is a single letter $a$; the general case follows by iteration. Apply induction on $k$, noting that $\rho_1$ counts distinct letters, so $\rho_1(v)\le\rho_1(av)$. 
For ${k\ge 2}$, assume $\rho_k(v) > 0$ (else the claim is trivial), and compare~$(\dagger)$, with $k$ replaced by ${k-1}$, for~$av$ and for $v$.
For each letter $b\ne a$, the $b$-tail is the same in $av$ and in $v$, so these letters contribute identically to both averages. The letter $a$ contributes $\rho_{k-1}(v)$ with weight $\binom{v}{k-2}$ to $\rho_k(av)$, and either nothing or $\rho_{k-1}(v_a)$ with weight~$\binom{v_a}{k-2}$ to $\rho_k(v)$. All tails involved are suffixes of~$v$, so by induction each term of positive weight in both averages is at most $\rho_{k-1}(v)$; also, $\binom{v_a}{k-2} \le \binom{v}{k-2}$, as~$v_a$ is a suffix of $v$. Thus replacing the value of the $a$-term in the average for $v$ by $\rho_{k-1}(v)$, then increasing its weight (from $0$, if the term is absent) to $\binom{v}{k-2}$, cannot decrease the average, proving the claim.

\medskip

By the claim, every term $\rho_k(w_\ell)$ of positive weight in $(\dagger)$ is at most $\rho_k(w)$, as each tail $w_\ell$ is a suffix of $w$. This proves that $\rho_{k+1}(w) \le \rho_k(w)$.
\end{proof}

\end{document}